\documentclass[a4paper,10pt]{article}
\usepackage[english,french]{babel}
\usepackage[T1]{fontenc}
\usepackage[utf8]{inputenc}
\usepackage{amsmath}
\usepackage{array}
\usepackage{amsthm}
\usepackage{amssymb}
\input xy
\xyoption{all}
\usepackage{hyperref}

\newcommand{\A}{{\mathcal{A}}}
\newcommand{\B}{{\mathcal{B}}}
\newcommand{\C}{{\mathcal{C}}}
\newcommand{\F}{{\mathcal{F}}}
\newcommand{\fct}{{\mathbf{Fct}}}
\newcommand{\Pp}{{\mathcal{P}}}
\newcommand{\pol}{{\mathcal{P}ol}}
\newcommand{\E}{{\mathcal{E}}}
\newcommand{\M}{{\mathcal{M}}}
\newcommand{\col}{{\rm colim}\,}
\newcommand{\Md}{\text{-}\mathbf{Mod}}
\newcommand{\Si}{\mathfrak{S}}
\newcommand{\pf}{\mathbf{pf}}

\title{Sur la noethérianité locale des foncteurs polynomiaux}

\author{Aur\'elien Djament\thanks{CNRS, LAGA (UMR 7539), Institut Galilée,
99 avenue J.-B. Clément,
93430 VILLETANEUSE,
FRANCE, aurelien.djament@cnrs.fr.}\; et Antoine Touz\'e\thanks{Université de Lille, laboratoire Paul Painlev\'e (UMR 8524),
Cité scientifique, bât. M2,
59655 VILLENEUVE D'ASCQ CEDEX,
FRANCE, antoine.touze@univ-lille.fr.}}

\date{Mai 2023}

\newtheorem{thi}{Th\'eor\`eme}

\newtheorem{thm}{Th\'eor\`eme}[section]
\newtheorem{pr}[thm]{Proposition}

\newtheorem{lm}[thm]{Lemme}

\theoremstyle{definition}
\newtheorem{defi}[thm]{D\'efinition}
\newtheorem{nota}[thm]{Notation}

\newtheorem*{conv}{Conventions}

\theoremstyle{remark}
\newtheorem{rem}[thm]{Remarque}
\newtheorem{app}[thm]{Application}

\newtheorem*{rqi}{Remarque}

\begin{document}

\maketitle

\begin{abstract}
Soient $A$ un anneau commutatif de type fini et $k$ un anneau commutatif noethérien. On montre que, dans la catégorie des foncteurs des $A$-modules projectifs de type fini vers les $k$-modules, tout foncteur polynomial de type fini est noethérien et possède une résolution projective de type fini.
\end{abstract}

{\selectlanguage{english}
{\begin{abstract}
Let $A$ be a finitely-generated commutative ring and $k$ a noetherian commutative ring. We show that, in the category of functors from finitely-generated projective $A$-modules to $k$-modules, each finitely-generated polynomial functor is noetherian and has a finitely-generated projective resolution.
\end{abstract}
}}

\bigskip

\noindent
{\em Mots-clefs : } foncteurs polynomiaux, foncteurs polynomiaux stricts, noethérianité, résolutions projectives de type fini.

\medskip

\noindent
{\em Classification MSC 2020 : } 16P40, 18A25, 18E05, 18G10, 20G43.


\section*{Introduction}

Introduits par Eilenberg et MacLane \cite{EML} au début des années 1950 à des fins de topologie algébrique, les \emph{foncteurs polynomiaux} entre catégories de modules n'ont cessé de montrer leur intérêt dans plusieurs branches des mathématiques : non seulement en topologie algébrique (avec de nouvelles motivations apparues à partir de la fin des années 1980 --- cf. \cite{HLS}), mais aussi en théorie des représentations \cite{Green,MD-sym} ou en cohomologie des groupes \cite{FFSS,Betley,Sco,DV,DjaR}. Des liens féconds ont été mis en évidence, à partir des années 1990, entre les foncteurs polynomiaux \emph{stricts} (dont la définition est donnée dans un contexte très général au §\,\ref{snPp}) et les représentations et la cohomologie des groupes algébriques \cite{FS,SFB,T-Duke2010,TvdK,T-Adv2010}. Intuitivement, les foncteurs polynomiaux ordinaires (à la Eilenberg-MacLane) sont aux foncteurs polynomiaux stricts ce que les fonctions polynomiales sont aux polynômes formels. Un foncteur d'oubli exact et fidèle associe à tout foncteur polynomial strict un foncteur polynomial ordinaire, mais même entre catégories de modules sur des anneaux très raisonnables, une structure polynomiale \emph{stricte} sur un foncteur polynomial ordinaire n'existe pas toujours (et n'est pas nécessairement unique). Les foncteurs polynomiaux stricts possèdent des propriétés de structure plus rigides et plus accessibles que les foncteurs polynomiaux ordinaires, et leur cohomologie est plus facile à calculer, donnant lieu à des interactions fécondes ---  cf. \cite{FFSS,T-torsion}.

Dans le présent article, on s'intéresse aux propriétés de finitude\,\footnote{Pour des rappels sur les propriétés de finitude utiles dans les catégories de foncteurs, on pourra consulter par exemple \cite[partie~IV]{DTV} et \cite{DT-schw}.}, notamment à la propriété \emph{noethérienne} (c'est-à-dire la stabilisation des suites croissantes de sous-objets), des foncteurs polynomiaux, ordinaires ou stricts, entre catégories de modules sur des anneaux raisonnables. 

Un théorème remarquable de Putman, Sam et Snowden \cite{PSam,SamSn} montre que la catégorie $\F(A,k)$ de \emph{tous} les foncteurs (sans condition polynomiale) depuis les $A$-modules projectifs de type fini sur un anneau $A$ vers les $k$-modules est \emph{localement noethérienne} (c'est-à-dire engendrée par un ensemble d'objets noethériens) si $A$ est un anneau \emph{fini} et $k$ un anneau noethérien. En revanche, si l'anneau $A$ est infini et l'anneau $k$ non nul, il n'est pas difficile de voir \cite[prop.~11.1]{DTV} que la catégorie $\F(A,k)$ n'est \emph{jamais} localement noethérienne.

La sous-catégorie $\pol(A,k)$ des foncteurs \emph{polynomiaux} de $\F(A,k)$ possède bien plus souvent que $\F(A,k)$ de bonnes propriétés de finitude. Il est ainsi facile de voir que $\pol(A,k)$ est \emph{localement finie} lorsque l'anneau $A$ est fini et que $k$ est un corps, et \emph{localement noethérienne} si le groupe additif sous-jacent à $A$ est de type fini et l'anneau $k$ noethérien, juste en examinant l'évaluation des foncteurs sur $A^n$ pour $n$ assez grand (cf. \cite[lemme~11.10]{DTV} ou \cite[prop.~4.8]{Dja-FM}). Toutefois, $\pol(A,k)$ n'est généralement pas localement noethérienne, même si $A$ et $k$ sont des anneaux noethériens, et ce type de problème de noethérianité ne semble guère avoir été étudié jusqu'alors.

L'un des résultats principaux du présent article est le théorème suivant :
\begin{thi}~\label{thi1} Si $A$ est un anneau commutatif de type fini et $k$ un anneau commutatif noethérien, alors la catégorie $\pol(A,k)$ est localement noethérienne.
\end{thi}
Le théorème précédent est une version légèrement simplifiée du théorème~\ref{th-noethbis}  du corps de l'article. Il s'obtient à partir de trois ingrédients différents :
\begin{enumerate}
\item un résultat analogue de noethérianité pour des foncteurs polynomiaux \emph{stricts}, le théorème~\ref{th-pstrnoeth} (cela fait l'objet du §\,\ref{snstr}). Celui-ci se déduit du théorème d'E. Noether de préservation par invariants sous l'action d'un groupe fini de la propriété de type fini pour les algèbres commutatives sur un anneau commutatif noethérien.
\item Des résultats de notre travail récent \cite{DT-schw} donnant des conditions suffisantes pour qu'un foncteur de $\pol(A,k)$ possède une résolution projective de type fini dans $\F(A,k)$, dont on déduit au §\,\ref{snpfi} un critère pour que l'image dans $\F(A,k)$ d'un foncteur polynomial \emph{strict} de type fini possède une résolution projective de type fini.
\item Un théorème fondamental de structure des foncteurs polynomiaux dû à Pirashvili \cite{Pira88} (rappelé à la fin du §\,\ref{sect-fct}) qui implique que, \emph{modulo des foncteurs polyomiaux de degré strictement inférieur} à $d$, tout foncteur polynomial de degré $d$ de $\pol(A,k)$ possède une structure polynomiale \emph{stricte}.
\end{enumerate}

Le deuxième ingrédient nous permet d'établir, \emph{en même temps que le théorème~\ref{thi1}} (les deux résultats s'obtiennent de façon imbriquée par une récurrence sur le degré polynomial), le résultat suivant, qui  constitue une version légèrement simplifiée du théorème~\ref{th-princ} donné en fin d'article :
\begin{thi}~\label{thi2} Si $A$ est un anneau commutatif de type fini et $k$ un anneau commutatif noethérien, alors tout foncteur de type fini de $\pol(A,k)$ possède une résolution projective de type fini dans $\F(A,k)$.
\end{thi}

Le théorème~\ref{thi2} est significatif au vu des liens entre l'homologie des groupes linéaires et l'homologie des foncteurs : ceux-ci concernent des foncteurs polynomiaux, mais font apparaître de l'algèbre homologique dans la catégorie $\F(A,k)$.

Le troisième ingrédient susmentionné est lié à la notion de \emph{recollement} de catégories abéliennes (cf. par exemple \cite{Psa}), qui est essentiellement équivalente à l'étude des quotients d'une catégorie abélienne par une sous-catégorie bilocalisante et des adjonctions associées. Il ne semble toutefois pas exister de critère général simple pour qu'un recollement de deux catégories abéliennes localement noethériennes soit localement noethérien. Dans le cas du recollement de Pirashvili associé aux foncteurs polynomiaux de degré au plus $d$, c'est une propriété bien particulière des adjonctions associées qui permet de conclure (en utilisant les deux autres ingrédients). En effet, elles s'obtiennent à partir de l'adjonction entre effets croisés et précomposition par la diagonale itérée, qui sont des foncteurs \emph{exacts}, et de l'action d'un groupe \emph{fini} (le groupe symétrique $\Si_d$). Ces arguments sont mis en \oe uvre dans la section~\ref{snprinc}, qui contient les principaux théorèmes de l'article.

\begin{rqi}
Malgré leurs titre et sujet similaires, le présent article et \cite{Dja-FM}, qui traite principalement de foncteurs polynomiaux sur des catégories \emph{non additives}, se recoupent très peu tant dans leurs résultats que leurs méthodes.

L'article de Draisma \cite{Dra} montre pour sa part un théorème de noethérianité \emph{topologique} pour des foncteurs polynomiaux \emph{stricts} sur un corps commutatif infini. Là encore, les méthodes et résultats de ce travail et du nôtre sont indépendants.
\end{rqi}

\paragraph*{Remerciements.} Les auteurs ont bénéficié du soutien partiel de l’Agence Nationale de la Recherche, via le projet ANR ChroK (ANR-16-CE40-0003), le Labex CEMPI (ANR-11-LABX-0007-01), et, pour le premier auteur, le projet ANR AlMaRe (ANR-19-CE40-0001-01). Ils ne soutiennent pas pour autant le principe de l’ANR, dont ils revendiquent la restitution des moyens aux laboratoires sous forme de crédits récurrents.

\begin{conv}
Dans tout cet article, $k$ désigne un anneau commutatif.

Sauf mention du contraire, les produits tensoriels de base non spécifiée sont pris sur $k$.

On note $\mathbf{Ab}$ la catégorie des groupes abéliens. Si $A$ est un anneau, on note $A\Md$ la catégorie des $A$-modules à gauche et $\mathbf{P}(A)$ la catégorie des $A$-modules \emph{à droite} projectifs de type fini.

On note $k[-]$ le foncteur de $k$-linéarisation des ensembles vers $k\Md$.
\end{conv}

\section{Foncteurs polynomiaux ordinaires}\label{sect-fct}

Si $\C$ et $\E$ sont des catégories, avec $\C$ essentiellement petite, on note $\fct(\C,\E)$ la catégorie des foncteurs de $\C$ vers $\E$. Si $\E$ est abélienne, il en est de même pour $\fct(\C,\E)$. Nous nous intéresserons principalement au cas où $\E=k\Md$, auquel cas on notera $\F(\C;k)$ pour $\fct(\C,k\Md)$, et plus particulièrement au cas où $\C=\mathbf{P}(A)$ pour un anneau $A$ --- on notera $\F(A,k)$ pour $\F(\mathbf{P}(A);k)$. On renvoie par exemple à \cite[§\,1.2]{DTV} pour quelques généralités classiques sur ces catégories.

Si $\A$ est une catégorie préadditive essentiellement petite et $\E$ une catégorie abélienne, on note $\mathbf{Add}(\A,\E)$ la catégorie (abélienne) des foncteurs additifs de $\A$ vers $\E$, et $\mathbf{Add}(\A;k):=\mathbf{Add}(\A,k\Md)$. Ainsi, $\mathbf{Add}(\mathbf{P}(A),\mathbf{Ab})\simeq A\Md$. Si de plus $\A$ et $\E$ sont $k$-linéaires, on note $\mathbf{Add}_k(\A,\E)$ la sous-catégorie (abélienne) des foncteurs $k$-linéaires de $\mathbf{Add}(\A,\E)$. Cette situation généralise la précédente au vu de l'isomorphisme canonique $\fct(\C,\E)\simeq\mathbf{Add}_k(k[\C],\E)$, où $k[\C]$ désigne la catégorie $k$-linéaire ayant les mêmes objets que $\C$ et dont les $k$-modules de morphismes s'obtiennent en linéarisant les ensembles de morphismes de $\C$.

Si $A$ est une $k$-algèbre et $\A$ une catégorie préadditive $k$-linéaire, la catégorie $\mathbf{Add}_k(\A,A\Md)$ est engendrée par les foncteurs $\A(a,-)\otimes A$ où $a$ parcourt les objets de $\A$. Ces foncteurs sont projectifs et de type fini. Un foncteur $F$ de $\mathbf{Add}_k(\A,A\Md)$ est de type fini si et seulement s'il est quotient d'une somme directe finie de tels foncteurs. (On rappelle qu'un objet $X$ d'une catégorie abélienne cocomplète $\E$ est dit \emph{de type fini} si l'application canonique $$\underset{J}{\col}\mathrm{Hom}(X,-)\circ\Phi\to\mathrm{Hom}(X,\underset{J}{\col}\Phi)$$
est injective pour tout foncteur $\Phi : J\to\E$, où $J$ est une petite catégorie \emph{filtrante}.)

Si l'on suppose que $\A$ est une catégorie \emph{additive} (essentiellement petite) et que $\E$ est une catégorie abélienne, on dispose de la notion fondamentale de \emph{foncteur polynomial} dans $\fct(\A,\E)$, due à Eilenberg et MacLane \cite[chap.~II]{EML}. Celle-ci est définie à partir des \emph{effets croisés}, c'est-à-dire des foncteurs $cr_n : \fct(\A,\E)\to\fct(\A^n,\E)$ définis (pour $n\in\mathbb{N}$) par
$$cr_n(F)(a_1,\dots,a_n)=\operatorname{Ker}\Big(F(a_1\oplus\dots\oplus a_n)\to\bigoplus_{i=1}^n F\Big(\bigoplus_{j\ne i}a_j\Big)\Big)$$
(les morphismes étant induits par les projections canoniques). Les foncteurs polynomiaux de degré au plus $d$ sont par définition les foncteurs $F$ tels que $cr_{d+1}(F)=0$. Ils forment une sous-catégorie épaisse, stable par limites et colimites, de $\fct(\A,\E)$, notée $\pol_d(\A,\E)$.  On note aussi $\pol_d(A,B):=\pol_d(\mathbf{P}(A),B)$, quand $A$ et $B$ sont des anneaux.

On renvoie à \cite[§\,2]{DTV} pour davantage de rappels sur les foncteurs polynomiaux. Nous aurons notamment besoin des notions et résultat suivants, dus à Pirashvili \cite{Pira88}, pour lesquels on suit la présentation de \cite[§\,2.6]{DTV} (auquel on renvoie pour plus de détails). La restriction de $cr_d$ à $\pol_d(\A,\E)$ prend ses valeurs dans la sous-catégorie\,\footnote{Si $\A$ et $\B$ sont des catégories préadditives, $\A\otimes_\mathbb{Z}\B$ désigne la catégorie préadditive ayant les mêmes objets que $\A\times\B$ et dont les morphismes sont donnés par $(\A\otimes_\mathbb{Z}\B)((a,b),(a',b')):=\A(a,a')\otimes_\mathbb{Z}\B(b,b')$.} $\mathbf{Add}_d(\A,\E)\simeq\mathbf{Add}(\A^{\otimes_\mathbb{Z} d},\E)$ de $\fct(\A^d,\E)$ des multifoncteurs additifs par rapport à chacune des $d$ variables. De plus $cr_d(F)$ est symétrique en ces variables : on dispose d'isomorphismes naturels $cr_d(F)(a_{\sigma(1)},\dots,a_{\sigma(1)})\simeq cr_d(F)(a_1,\dots,a_d)$ pour toute permutation $\sigma\in\Si_d$, compatibles à la composition des permutations ; on note $\Sigma\mathbf{Add}_d(\A,\E)$ comme dans \cite[§\,2.6]{DTV} la catégorie des multifoncteurs de $\mathbf{Add}_d(\A,\E)$ munis de tels isomorphismes, et $\mathrm{Cr}_d : \pol_d(\A,\E)\to\Sigma\mathbf{Add}_d(\A,\E)$ le foncteur induit par $cr_d$. On note 
\begin{equation}\label{eq-D}
    \Delta_d : \Sigma\mathbf{Add}_d(\A,\E)\to\pol_d(\A,\E)
\end{equation}
la composée du foncteur d'oubli $\Sigma\mathbf{Add}_d(\A,\E)\to\mathbf{Add}_d(\A,\E)$ et de la précomposition par le foncteur de diagonale $d$-itérée $\A\to\A^d$. La structure symétrique sur les objets de $\Sigma\mathbf{Add}_d(\A,\E)$ munit $\Delta_d$ d'une action du groupe symétrique $\Si_d$. L'énoncé suivant est une reformulation d'une partie du théorème fondamental de Pirashvili \cite[th.~2.1]{Pira88}\,\footnote{Pirashvili donne le résultat dans le cas où $\A=\mathbf{P}(A)$ et $\E=\mathbf{Ab}$, mais la même démonstration vaut dans le cas général.} identifiant la catégorie quotient $\pol_d(\A,\E)/\pol_{d-1}(\A,\E)$ à $\Sigma\mathbf{Add}_d(\A,\E)$.

\begin{pr}[Pirashvili]\label{pr-pira} Soient $\A$ une catégorie additive essentiellement petite et $\E$ une catégorie abélienne. Le foncteur $\mathrm{Cr}_d : \pol_d(\A,\E)\to\Sigma\mathbf{Add}_d(\A,\E)$ est adjoint à gauche au foncteur $\Delta_d^{\Si_d}$, où l'exposant indique les invariants sous l'action du groupe $\Si_d$. De plus, pour tout foncteur $F$ de $\pol_d(\A,\E)$, le noyau et le conoyau de l'unité $F\to\Delta_d^{\Si_d}\mathrm{Cr}_d(F)$ appartiennent à $\pol_{d-1}(\A,\E)$.
\end{pr}

\section{Foncteurs polynomiaux stricts}\label{snPp}

Rappelons \cite[chap.~III, §\,1]{Rob} que, si $V$ est un $k$-module, on définit l'algèbre des \emph{puissances divisées} de $V$ sur $k$, notée $\Gamma^*_k(V)$ ou simplement $\Gamma^*(V)$ s'il n'y a pas d'ambiguïté possible sur l'anneau de base $k$, comme la $k$-algèbre associative, commutative et unitaire engendrée par des éléments $v^{[n]}$ pour $v\in V$ et $n\in\mathbb{N}$ soumis aux relations suivantes :
\begin{itemize}
    \item $\forall v\in V,\qquad v^{[0]}=1$ ;
    \item $\forall (\lambda,v,n)\in k\times V\times\mathbb{N},\qquad (\lambda v)^{[n]}=\lambda^n.v^{[n]}$ ;
    \item $\forall (v,n,m)\in V\times\mathbb{N}\times\mathbb{N},\qquad v^{[n]}.v^{[m]}=\frac{(n+m)!}{n! m!}.v^{[n+m]}$ ;
    \item $\forall (v,w,n)\in V\times V\times\mathbb{N},\qquad (v+w)^{[n]}=\sum_{i+j=n}v^{[i]}.w^{[j]}\;$.
\end{itemize}
Cette algèbre est graduée par $\deg(v^{[n]}):=n$. On note $\Gamma^d_k(V)$, ou $\Gamma^d(V)$, la composante homogène de degré $d$ de l'algèbre graduée $\Gamma^*(V)$. On définit ainsi un endofoncteur $\Gamma^d$ des $k$-modules appelé $d$-ième puissance divisée. On renvoie à \cite[chap.~III]{Rob} pour les propriétés de base de ces foncteurs. Rappelons simplement que $\Gamma^d$ est un foncteur polynomial de degré $d$, qui préserve les épimorphismes et les colimites filtrantes, et que l'on dispose d'un morphisme naturel $\Gamma^d(V)\to (V^{\otimes d})^{\Si_d}$ qui est un isomorphisme si $V$ est un module \emph{plat}. Le foncteur $\Gamma^d_k$ est compatible au changement de l'anneau de base $k$ en le sens suivant : si $k\to A$ est un morphisme d'anneaux, avec $A$ commutatif, on dispose d'un isomorphisme $A$-linéaire $\Gamma^d_A(A\otimes_k V)\simeq A\otimes_k\Gamma^d_k(V)$ naturel en le $k$-module $V$ \cite[th.~III.3]{Rob}.

Le foncteur $\Gamma^d_k$ possède une structure canonique d'endofoncteur \emph{monoïdal symétrique} de la catégorie monoïdale symétrique $(k\Md,\otimes,k)$. Cela permet de relever $\Gamma^d_k(A)$ en un endofoncteur des $k$-algèbres (voir \cite{Rob80} pour une description détaillée). De même, si $\A$ est une catégorie $k$-linéaire, on définit une catégorie $k$-linéaire $\Gamma^d_k\A$ ayant les mêmes objets que $\A$ et dont les morphismes sont donnés par $(\Gamma^d_k\A)(a,b):=\Gamma^d_k(\A(a,b))$ et la composition par
$$\Gamma^d_k(\A(a,b))\otimes\Gamma^d_k(\A(b,c))\to\Gamma^d_k(\A(a,b)\otimes\A(b,c))\to\Gamma^d_k(\A(a,c))$$
où la première flèche est induite par la structure monoïdale de $\Gamma^d_k$ et la seconde par la composition de $\A$.

\begin{defi}
Soient $\A$ une petite catégorie $k$-linéaire, $\E$ une catégorie abélienne $k$-linéaire et $d\in\mathbb{N}$. On note $\Pp_{d;k}(\A,\E)$ la catégorie $\mathbf{Add}_k(\Gamma^d_k(\A),\E)$. Ses objets sont appelés \emph{foncteurs polynomiaux stricts} homogènes de degré $d$ sur $k$ de $\A$ vers $\E$.

On note $\Pp_{d;k}(\A;k)$ pour $\Pp_{d;k}(\A,k\Md)$ ; lorsque $A$ et $B$ sont des $k$-algèbres, on note $\Pp_{d;k}(A,B)$ pour $\Pp_{d;k}(\mathbf{P}(A),B\Md)$.
\end{defi}

Cette notion a été introduite originellement, dans le cas de $\Pp_{d;k}(k,k)$, par Friedlander et Suslin \cite[§\,2]{FS} (lorsque $k$ est un corps, généralisé aussitôt au cas d'un anneau commutatif quelconque dans \cite{SFB}). On pourra également consulter \cite[§\,2.1]{TouzeRingel} pour des généralités sur ces catégories.

\begin{rem}
On peut définir des foncteurs polynomiaux stricts \guillemotleft\, naïfs \guillemotright\, de la même manière en remplaçant le foncteur $\Gamma^d$ par le foncteur des tenseurs symétriques homogènes de degré $d$ (invariants de la $d$-ième puissance symétrique sous l'action de $\Si_d$). Les deux notions coïncident si $\A$ est $k$-plate, i.e. si tous les $k$-modules $\A(a,b)$ sont plats. La plupart des résultats que nous démontrerons pour les foncteurs polynomiaux stricts valent également pour les foncteurs polynomiaux stricts naïfs, mais ces derniers sont bien moins importants : en l'absence d'hypothèse de platitude, les tenseurs symétriques ont beaucoup moins de bonnes propriétés fonctorielles que les puissances divisées, qui sont reliées à la géométrie algébrique \cite{SFB}.
\end{rem}

Le résultat suivant se déduit du théorème de Gabriel-Popescu, de la même façon que dans le cas particulier classique où $A=B=k$ est un corps \cite[th.~3.2]{FS}.

\begin{pr}\label{pr-eqSchur} Soient $A$ et $B$ des $k$-algèbres et $n\ge d\ge 0$ des entiers. Alors le foncteur d'évaluation en $A^d$ induit une équivalence de catégories $k$-linéaire
$$\Pp_{d;k}(A,B)\xrightarrow{\simeq}(\Gamma^d_k(\M_n(A))\otimes B)\Md.$$
\end{pr}

\begin{rem}
Les algèbres $\Gamma^d_k(\M_n(A))$ sont des généralisations des algèbres de Schur classiques.

On peut également montrer que la catégorie $\pol_d(A,B)$ est équivalente à une catégorie de modules sur un certain anneau \cite[cor.~1.1]{Pira88}, mais cet anneau est peu maniable, de sorte que ce dernier résultat ne nous aidera pas dans le présent travail, contrairement à la proposition~\ref{pr-eqSchur}.
\end{rem}

On dispose d'un morphisme $k$-linéaire $k[V]\to\Gamma^d_k(V)\quad [v]\mapsto v^{[d]}$ naturel en le $k$-module $V$ et compatible aux structures monoïdales symétriques. On obtient ainsi un foncteur $k$-linéaire $k[\A]\to\Gamma^d_k(\A)$ qui induit par précomposition un foncteur $k$-linéaire
$$i_d : \Pp_{d;k}(\A,\E)=\mathbf{Add}_k(\Gamma^d_k(\A),\E)\to\mathbf{Add}_k(k[\A],\E)\simeq\fct(\A,\E)\;.$$
Celui-ci, comme tout foncteur de précomposition, commute aux limites et aux colimites (et est en particulier exact). Comme $k[\A]\to\Gamma^d_k(\A)$ est essentiellement surjectif, $i_d$ est également \emph{fidèle}. Son image essentielle est incluse dans $\pol_d(\A,\E)$. Si $d!$ est inversible dans $k$, alors $i_d$ est plein et son image essentielle est stable par sous-quotient, mais ce n'est pas vrai en général.

\begin{pr}\label{pr-PFtf} Soient $\A$ une catégorie additive essentiellement petite et $\E$ une catégorie de Grothendieck, toutes deux $k$-linéaires. Un foncteur $F$ de $\Pp_{d;k}(\A,\E)$ est de type fini si et seulement si le foncteur $i_d(F)$ de $\fct(\A,\E)$ est de type fini.
\end{pr}

\begin{proof}
Si $V$ est un $k$-module, le morphisme $k$-linéaire naturel $$\underset{i_1,\dots,i_r>0}{\bigoplus_{i_1+\dots+i_r=d}}k[V^r]\to\Gamma^d_k(V)\qquad [v_1,\dots,v_r]\mapsto v_1^{[i_1]}\dots v_r^{[i_r]}$$
est surjectif. On en déduit, pour tout objet $a$ de $\A$, une transformation naturelle surjective
$$\underset{i_1,\dots,i_r>0}{\bigoplus_{i_1+\dots+i_r=d}}k[\A](a^{\oplus r},-)\twoheadrightarrow i_d\big(\Gamma^d\A(a,-)\big)\,.$$
Il s'ensuit que l'adjoint à droite $j_d$ de $i_d$ (qui existe par la théorie générale des extensions de Kan) est muni pour tout objet $a$ de $\A$ de monomorphismes
$$j_d(F)(a)\hookrightarrow\underset{i_1,\dots,i_r>0}{\bigoplus_{i_1+\dots+i_r=d}}F(a^{\oplus r})$$
naturels en le foncteur $F$ de $\fct(\A,\E)$. Par conséquent, si $\Phi$ est un foncteur depuis une petite catégorie \emph{filtrante} $J$ vers $\fct(\A,\E)$, le morphisme canonique $\underset{J}{\col}j_d\circ\Phi\to j_d\Big(\underset{J}{\col}\Phi\Big)$ est un monomorphisme. Cela entraîne formellement que le foncteur $i_d$ préserve les objets de type fini.

La réciproque découle de ce que $i_d$ est exact et fidèle et commute aux colimites.
\end{proof}

On dispose d'un autre foncteur fondamental reliant les foncteurs polynomiaux stricts aux foncteurs (multiadditifs) ordinaires. Pour alléger les notations, nous ne l'explicitons que dans le cas de l'anneau de base des entiers, qui seul nous servira, en lien avec la proposition~\ref{pr-pira}.

Si $\A$ est une catégorie additive essentiellement petite et $\E$ une catégorie abélienne, la précomposition par le foncteur additif canonique $\Gamma^d_\mathbb{Z}(\A)\to\A^{\otimes d}$ (où la puissance tensorielle est prise sur $\mathbb{Z}$) fournit un foncteur exact $\mathbf{Add}_d(\A,\E)\simeq\mathbf{Add}(\A^{\otimes d},\E)\to\mathbf{Add}(\Gamma^d_\mathbb{Z}(\A),\E)=\Pp_{d;\mathbb{Z}}(\A,\E)$. La précomposition par le foncteur d'oubli $\Sigma\mathbf{Add}_d(\A,\E)\to\mathbf{Add}_d(\A,\E)$ fournit un foncteur
\begin{equation}\label{eq-Dt}
    \tilde{\Delta}_d : \Sigma\mathbf{Add}_d(\A,\E)\to\Pp_{d;\mathbb{Z}}(\A,\E)
\end{equation}
muni d'une action de $\Si_d$, qui relève le foncteur $\Delta_d : \Sigma\mathbf{Add}_d(\A,\E)\to\pol_d(\A,\E)$ introduit en \eqref{eq-D} à la fin de la section~\ref{sect-fct}, au sens où l'on dispose d'un isomorphisme $\Si_d$-équivariant $\Delta_d\simeq i_d\circ\tilde{\Delta}_d$.

\section{Noethérianité locale des foncteurs polynomiaux stricts}\label{snstr}

Le théorème principal de cette section repose sur le résultat classique suivant.

\begin{pr}[Noether]\label{pr-Noether} Supposons que $k$ est noethérien et que $A$ est une $k$-algèbre commutative de type fini, munie d'une action d'un groupe fini $G$ par automorphismes d'algèbre. Alors la $k$-algèbre $A^G$ des invariants est de type fini.

Si de plus $V$ est un $A$-module de type fini muni d'une action de $G$ par automorphismes $k$-linéaires tels que $g_*(a.v)=(g_*a).(g_*v)$ pour tout $(g,a,v)\in G\times A\times V$ (où l'étoile indique l'action de $G$ sur $A$ ou $V$ et le point l'action de $A$ sur $V$), alors $V^G$ est un $A^G$-module de type fini.
\end{pr}

\begin{proof} La première assertion est un théorème classique de Noether, cf. par exemple \cite[chap.~V, §\,1, th.~2]{Bki-AC2}. Pour établir la deuxième, on note d'abord que, si $A$ est une $k$-algèbre commutative de type fini et $V$ un $A$-module, alors $V$ est un $A$-module de type fini si et seulement si la $k$-algèbre commutative $V\rtimes A:=V\oplus A$ (où $V$ est un idéal de carré nul) est de type fini. Or, si $V$ est muni d'une action de $G$ vérifiant la condition de l'énoncé, on dispose d'un isomorphisme de $k$-algèbres $(V\rtimes A)^G\simeq V^G\rtimes A^G$, de sorte que la deuxième assertion s'obtient en appliquant la première à $V\rtimes A$.
\end{proof}

On rappelle qu'une $k$-algèbre commutative est dite \emph{essentiellement de type fini} si elle est isomorphe à la localisation relativement à une partie multiplicative d'une $k$-algèbre commutative de type fini (par exemple, un corps commutatif de type fini est une algèbre commutative essentiellement de type fini sur $\mathbb{Z}$). Une $k$-algèbre est dite \emph{finie} si le $k$-module sous-jacent est de type fini.

\begin{thm}\label{th-pstrnoeth}
Supposons que l'anneau $k$ est noethérien, que $A$ est une $k$-algèbre commutative essentiellement de type fini et que $L$ est une $A$-algèbre finie. Alors pour tout $d\in\mathbb{N}$, la catégorie de foncteurs polynomiaux stricts $\Pp_{d;k}(L,k)$ est localement noethérienne.
\end{thm}

\begin{proof} En vertu de la proposition~\ref{pr-eqSchur}, il s'agit de montrer que $\Gamma^d(\M_d(L))$ est une algèbre noethérienne à gauche.

On commence par le faire dans le cas où $A$ est une $k$-algèbre de type fini et plate. Dans ce cas $\Gamma^d(A)\simeq (A^{\otimes d})^{\Si_d}$ est une algèbre de type fini, et $(V^{\otimes d})^{\Si_d}$ est un $\Gamma^d(A)$-module de type fini pour tout $A$-module de type fini $V$, par la proposition~\ref{pr-Noether}. Il s'ensuit que $\Gamma^d(V)$ est également un $\Gamma^d(A)$-module de type fini, car tout $A$-module de type fini est quotient d'un $A$-module libre de type fini, qui est $k$-plat, et $\Gamma^d$ préserve les épimorphismes. Or $\M_d(L)$ est un $A$-module de type fini, donc $\Gamma^d(\M_d(L))$ est une algèbre finie sur l'algèbre commutative de type fini $\Gamma^d(A)$, c'est donc une algèbre noethérienne (à gauche et à droite).

Il existe par hypothèse $n\in\mathbb{N}$, un idéal $I$ de l'algèbre de polynômes $k[x_1,\dots,x_n]$ et une partie multiplicative $S$ de l'algèbre quotient $k[x_1,\dots,x_n]/I$ tels que $A\simeq\big(k[x_1,\dots,x_n]/I)[S^{-1}]$. L'image inverse $\tilde{S}$ de $S$ par la projection $k[x_1,\dots,x_n]\twoheadrightarrow k[x_1,\dots,x_n]/I$ est une partie multiplicative de $k[x_1,\dots,x_n]$, et l'on a $A\simeq\big(k[x_1,\dots,x_n][\tilde{S}^{-1}]\big)/I[\tilde{S}^{-1}]$. Supposons que $L$ est engendré comme $A$-module par une partie finie $E$. Soient $R$ la sous-$k[x_1,\dots,x_n]$-algèbre de $L$ engendrée par $E$ (où l'on voit $L$ comme $k[x_1,\dots,x_n]$-algèbre via le morphisme $k[x_1,\dots,x_n]\to A$) et $T$ l'image de $\tilde{S}$ dans $R$ : $T$ est une partie multiplicative \emph{centrale} de $R$, et $L\simeq R[T^{-1}]$.

Comme $k[x_1,\dots,x_n]$ est une $k$-algèbre commutative de type fini plate et $R$ une $k[x_1,\dots,x_n]$-algèbre finie, le début de la démonstration montre que $\Gamma^d(\M_d(R))$ est une algèbre noethérienne. La conclusion provient alors des isomorphismes canoniques
$$\Gamma^d(\M_d(L))\simeq\Gamma^d(\M_d(R)[T^{-1}])\simeq\Gamma^d(\M_d(R))[(T^{[d]})^{-1}],$$
où $T^{[d]}$ désigne la partie multiplicative centrale $\{t^{[d]}\,\vert\,t\in T\}$ de $\Gamma^d(R)$.
\end{proof}

\begin{rem} Plus généralement, si $k$, $A$ et $L$ vérifient les hypothèses du théorème~\ref{th-pstrnoeth}, alors la catégorie $\Pp_{d;k}(\mathbf{P}(A),\E)$ est localement noethérienne pour toute catégorie de Grothendieck $k$-linéaire localement noethérienne $\E$.
\end{rem}

\section{La propriété $pf_\infty$}\label{snpfi}

Pour $n\in\mathbb{N}$, on dit qu'un objet $X$ d'une catégorie abélienne $\E$ engendrée par des objets projectifs de type fini vérifie la propriété de $n$-présentation finie --- en abrégé $pf_n$ --- s'il existe une suite exacte $P_n\to P_{n-1}\to\dots\to P_0\to X\to 0$ où les $P_i$ sont projectifs de type fini. On dit qu'un objet vérifie la propriété $pf_\infty$ s'il satisfait $pf_n$ pour tout $n\in\mathbb{N}$. On renvoie par exemple à \cite[§\,1]{DT-schw} pour plus de détails sur la propriété $pf_n$ ; on note $\pf_n(\E)$ la classe des objets de $\E$ la vérifiant.

Avant de donner le résultat principal de cette section, nous aurons besoin d'un lemme général qui permettra de contourner des problèmes posés par le comportement de la propriété $pf_n$ par changement de base \emph{non plat} en présence de propriétés de noethérianité.

\begin{lm}\label{lm-pfi_chgb} Supposons que $k$ est noethérien. Si $\C$ est une petite catégorie préadditive et $F$ un foncteur noethérien de $\mathbf{Add}(\C;\mathbb{Z})$, alors $\mathrm{Tor}_1^\mathbb{Z}(F,k)$ est un foncteur de type fini de $\mathbf{Add}(\C;k)$.
\end{lm}

\begin{proof}
Comme l'anneau $k$ est noethérien, l'idéal de ses éléments de $\mathbb{Z}$-torsion est de type fini, donc annulé par un entier $N>0$. Il s'ensuit que le morphisme naturel
$$\bigoplus_n\mathrm{Hom}_\mathbb{Z}(\mathbb{Z}/n,k)\otimes_\mathbb{Z}\mathrm{Tor}_1^\mathbb{Z}(F,\mathbb{Z}/n)\to\mathrm{Tor}_1^\mathbb{Z}(F,k)$$
de $\mathbf{Add}(\C;k)$ déduit de la fonctorialité de $\mathrm{Tor}_1^\mathbb{Z}$, où la somme directe est prise sur l'ensemble fini des diviseurs strictement positifs de $N$, est un épimorphisme. La conclusion provient donc de ce que $\mathrm{Hom}_\mathbb{Z}(\mathbb{Z}/n,k)\hookrightarrow k$ est un $k$-module de type fini et $\mathrm{Tor}_1^\mathbb{Z}(F,\mathbb{Z}/n)\simeq\mathrm{Hom}_\mathbb{Z}(\mathbb{Z}/n,F)\hookrightarrow F$ est un foncteur de type fini de $\mathbf{Add}(\C;\mathbb{Z})$, puisque $k$ et $F$ sont noethériens.
\end{proof}

\begin{thm}\label{th-pfiP} Soient $d\in\mathbb{N}$ et $\A$ une petite catégorie additive. On suppose que l'anneau $k$ est noethérien et que les catégories $\Pp_{d;\mathbb{Z}}(\A;k)$ et $\mathbf{Add}(\A;\mathbb{Z})$ sont localement noethériennes. 

Si $F$ est un foncteur de type fini de $\Pp_{d;\mathbb{Z}}(\A;k)$, alors $i_d(F)$ appartient à $\pf_\infty(\F(\A;k))$.
\end{thm}

\begin{proof}
Soit $a$ un objet de $\A$. Comme le foncteur $\A(a,-)$ est de type fini dans la catégorie localement noethérienne $\mathbf{Add}(\A;\mathbb{Z})$, son sous-foncteur de $\mathbb{Z}$-torsion est de type fini, de sorte qu'il est annulé par un entier $N>0$.

Notons $\mathbf{Ab}_N$ la sous-catégorie pleine de $\mathbf{Ab}$ constituée des groupes abéliens de type fini dont le sous-groupe de torsion est annulé par $N$. Cette catégorie est équivalente à $\mathbf{P}(\mathrm{A}_N)$, où $\mathrm{A}_N$ désigne l'anneau des endomorphismes du groupe abélien $\mathbb{Z}\oplus\bigoplus\mathbb{Z}/n$ où la somme directe est prise sur les diviseurs strictement positifs $n$ de $N$. Comme le groupe abélien sous-jacent à $\mathrm{A}_N$ est de type fini, \cite[cor.~6.18]{DT-schw} et la proposition~\ref{pr-PFtf} montrent que la restriction à $\mathbf{Ab}_N$ de $\Gamma^d_\mathbb{Z}$ appartient à $\pf_\infty(\F(\mathbf{Ab}_N;\mathbb{Z}))$.

Comme les endofoncteurs $\mathbb{Z}[-]$ et $\Gamma^d_\mathbb{Z}$ de $\mathbf{Ab}$ commutent aux colimites filtrantes, on en déduit que la restriction de $\Gamma^d_\mathbb{Z}$ aux groupes abéliens dont le sous-groupe de torsion est annulé par $N$ possède une résolution dont chaque terme est une somme directe finie de foncteurs de la forme $\mathbb{Z}[\mathbf{Ab}(V,-)]$, où $V$ appartient à $\mathbf{Ab}_N$. Il s'ensuit que $\Gamma^d_\mathbb{Z}(\A(a,-))$ possède une résolution dont chaque terme est une somme directe finie de foncteurs de la forme $\mathbb{Z}[A]$, où $A$ est un foncteur de type fini, donc vérifiant la propriété $pf_\infty$, de la catégorie localement noethérienne $\mathbf{Add}(\A;\mathbb{Z})$. Un tel foncteur $\mathbb{Z}[A]$ appartient à $\pf_\infty(\F(\A;\mathbb{Z}))$ en vertu de \cite[prop.~7.1]{DT-schw}. Par conséquent, $\Gamma^d_\mathbb{Z}(\A(a,-))$ appartient à $\pf_\infty(\F(\A;\mathbb{Z}))$. On conclut 
en montrant les propriétés $(H^1_n)$ et $(H^2_n)$ ci-dessous\\
$(H^1_n)\qquad\forall a\in\mathrm{Ob}\,\A\qquad\Gamma^d_\mathbb{Z}(\A(a,-))\otimes_\mathbb{Z}k\in\pf_n(\F(\A;k))$\\
$(H^2_n)\qquad$ pour tout $F\in\mathrm{Ob}\,\Pp_{d;\mathbb{Z}}(\A;k)$ de type fini, $i_d(F)\in\pf_n(\F(\A;k))$\\
selon le schéma de récurrence sur $n\in\mathbb{N}$ suivant :
$$(H^1_n)\Rightarrow (H^2_n)\Rightarrow (H^1_{n+1}).$$
L'initialisation $(H^1_0)$ est triviale, l'implication $(H^1_n)\Rightarrow (H^2_n)$ est formelle (cf. \cite[prop.~1.7]{DT-schw}), et l'implication $(H^2_n)\Rightarrow (H^1_{n+1})$ s'établit au moyen du lemme~\ref{lm-pfi_chgb} (appliqué à la catégorie préadditive $\C=\Gamma^d_\mathbb{Z}(\A)$), en utilisant \cite[cor.~1.6]{DT-schw} (et la proposition~\ref{pr-PFtf}), qu'on applique à la tensorisation par $k$ sur $\mathbb{Z}$ d'une résolution projective de type fini (qui est nécessairement $\mathbb{Z}$-libre) de $\Gamma^d_\mathbb{Z}(\A(a,-))$ dans $\F(\A;\mathbb{Z})$. Cela termine la démonstration.
\end{proof}

%
\section{Noethérianité locale des foncteurs polynomiaux ordinaires}\label{snprinc}

\begin{nota}
Si $A$ est un anneau, on désigne par $A_k$ la $k$-algèbre $A\otimes_\mathbb{Z}k$.
\end{nota}

\subsection{Premier théorème de finitude}

Le théorème suivant constitue une vaste généralisation de \cite[cor.~6.18]{DT-schw}. Sa démonstration utilise les foncteurs $\Delta_d$ et $\tilde{\Delta}_d$ introduits en \eqref{eq-D} et \eqref{eq-Dt} respectivement.

\begin{thm}\label{th-princ} Soit $A$ un anneau noethérien à gauche. On suppose que $k$ est noethérien et qu'il existe une $k$-algèbre commutative essentiellement de type fini $L$ telle que $A_k$ soit une $L$-algèbre finie. Alors tout foncteur polynomial de type fini $F$ de $\F(A,k)$ est noethérien et appartient à $\pf_\infty(\F(A,k))$.
\end{thm}

\begin{proof} On montre le résultat par récurrence sur le degré polynomial $d$ de $F$. Le résultat est trivial pour $d\le 0$, on peut donc supposer $d>0$ et l'assertion établie pour les foncteurs de $\pol_{d-1}(A,k)$. Comme $F$ est de type fini, le foncteur $\Delta_d\mathrm{Cr}_d(F)$ est également de type fini dans $\F(A,k)$ (ou $\pol_d(A,k)$). Il s'ensuit, grâce à la proposition~\ref{pr-PFtf}, que le foncteur $\tilde{\Delta}_d\mathrm{Cr}_d(F)$ est de type fini dans $\Pp_{d;\mathbb{Z}}(A,k)\simeq\Pp_{d;k}(A_k,k)$. Par le théorème~\ref{th-pstrnoeth}, cela entraîne que son sous-foncteur $\tilde{\Delta}_d^{\Si_d}\mathrm{Cr}_d(F)$ est de type fini. Le théorème~\ref{th-pfiP} montre alors que $\Delta_d^{\Si_d}\mathrm{Cr}_d(F)\simeq i_d\big(\tilde{\Delta}_d^{\Si_d}\mathrm{Cr}_d(F)\big)$
appartient à $\pf_\infty(\F(A,k))$.

Par la proposition~\ref{pr-pira}, il existe dans $\F(A,k)$ une suite exacte $0\to T\to F\to\Delta_d^{\Si_d}\mathrm{Cr}_d(F)\to X\to 0$
avec $T$ et $X$ dans $\pol_{d-1}(A,k)$. Quotient du foncteur de type fini $\Delta_d^{\Si_d}\mathrm{Cr}_d(F)$, $X$ est de type fini, et appartient donc à $\pf_\infty(\F(A,k))$ par l'hypothèse de récurrence. Le fait que $F$ soit de type fini, $\Delta_d^{\Si_d}\mathrm{Cr}_d(F)$ de présentation finie et $X$ dans $\pf_2(\F(A,k))$ implique formellement \cite[prop.~1.5]{DT-schw} que $T$ est de type fini. L'hypothèse de récurrence montre donc que $T$ appartient à $\pf_\infty(\F(A,k))$, et \cite[prop.~1.5]{DT-schw} que $F$ appartient également à cette classe de foncteurs.

Ainsi, tout foncteur de type fini $F$ de $\pol_d(A,k)$ est dans $\pf_\infty(\F(A,k))$, et en particulier de présentation finie dans $\F(A,k)$. Cela implique qu'un tel foncteur est noethérien : si $G$ est un sous-foncteur de $F$, alors $F/G$, quotient de $F$, est de type fini et dans $\pol_d(A,k)$, donc de présentation finie dans $\F(A,k)$, donc $G$ est nécessairement de type fini, d'où le théorème.
\end{proof}

\subsection{Deuxième théorème de noethérianité}

Nous allons voir que l'on peut s'affranchir de l'hypothèse de noethérianité de $A$ dans l'énoncé précédent pour obtenir la noethérianité locale de $\pol_d(A,k)$ (théorème~\ref{th-noethbis} ci-après), mais pas pour la propriété $pf_\infty$ (remarque ci-dessous).

\begin{rem}\label{rq-paspf2} Supposons que $k$ est un corps de caractéristique $p>0$ et considérons l'anneau commutatif $A=(\mathbb{Q}/\mathbb{Z})^{\oplus\mathbb{N}}\rtimes\mathbb{Z}$ extension de carré nul de $\mathbb{Z}$ par le groupe abélien $(\mathbb{Q}/\mathbb{Z})^{\oplus\mathbb{N}}$. Alors $A_k\simeq k$, de sorte que toutes les hypothèses du théorème~\ref{th-princ}, \emph{sauf la noethérianité de $A$}, sont vérifiées. Pour autant, l'idéal des éléments de $A$ annulés par $p$ n'est pas de type fini, de sorte que la catégorie additive $\mathbf{P}(A)$ ne vérifie pas la condition $(T_{2,p})$ de \cite[§\,5.1]{DT-schw}. Ainsi, \cite[th.~5.9]{DT-schw} montre qu'il existe un foncteur additif de type fini dans $\F(A,k)$ qui ne vérifie pas la propriété $pf_\infty$.
\end{rem}

La généralisation de la propriété noethérienne donnée par le théorème~\ref{th-princ} se déduira de celui-ci par l'intermédiaire de quelques lemmes simples que nous donnons maintenant. Les premiers relèvent de la seule théorie des anneaux.

Dans l'énoncé suivant, si $A$ est un anneau, on note $A_\mathrm{tor}$ l'idéal bilatère de $A$ constitué des éléments de $\mathbb{Z}$-torsion.

\begin{lm}\label{lm-strZan} Tout anneau commutatif noethérien $k$ est isomorphe au produit direct d'une $\mathbb{Z}[S^{-1}]$-algèbre $A$ telle que $A/A_\mathrm{tor}$ soit \emph{fidèlement plate} sur $\mathbb{Z}[S^{-1}]$, pour une certaine partie multiplicative $S$ de $\mathbb{Z}$, et d'un nombre fini d'anneaux dont la caractéristique est une puissance d'un nombre premier.
\end{lm}

\begin{proof} Si $k$ est de caractéristique non nulle, le résultat est clair, on suppose donc $k$ de caractéristique nulle. Soit $S$ la partie multiplicative de $\mathbb{Z}$ constituée des entiers qui sont inversibles dans l'anneau quotient $k/k_\mathrm{tor}$. Il est clair que $k/k_\mathrm{tor}$ est une $\mathbb{Z}[S^{-1}]$-algèbre fidèlement plate.

Si $p$ est un nombre premier appartenant à $S$, il existe $x\in k$ et des entiers $i\in\mathbb{N}$ et $m\in\mathbb{N}^*$, avec $m$ étranger à $p$, tels que $mp^i(1-px)=0$. Quitte à modifier $x$, on peut supposer que $m=1$ --- en effet, toute identité de Bézout $pa+mb=1$ implique $0=p^i(1-pa)(1-px)=p^i(1-py)$ avec $y=a+x-pax$. L'élément $p^i x^i$ de $k$ est alors idempotent, de sorte que $k$ est isomorphe au produit de $k/(p^i x^i)$, dont la caractéristique est une puissance de $p$ (car $p^i=p^{i+1}x$), et de $k/(1-p^i x^i)$, dans lequel $p$ est inversible. Comme $k_\mathrm{tor}$ n'a une composante $p$-primaire non nulle que pour un nombre fini de nombres premiers $p$ (puisque c'est un idéal de type fini de $k$), en itérant cette décomposition, on obtient le résultat souhaité.
\end{proof}

Les énoncés suivants, classiques, sont laissés en exercice.

\begin{lm}\label{lm-nfp} Soient $K\to L$ un morphisme d'anneaux commutatifs fidèlement plat et $A$ une $K$-algèbre. Si la $L$-algèbre $A\otimes_K L$ est noethérienne à gauche, alors $A$ est noethérienne à gauche.
\end{lm}

\begin{lm}\label{lm-relevatf} Soient $A$ un anneau et $I$ un idéal bilatère nilpotent de $A$, de type fini comme idéal à gauche. Si l'anneau $A/I$ est noethérien à gauche, alors $A$ est noethérien à gauche.
\end{lm}

Le lemme suivant s'établit facilement par récurrence sur le degré polynomial $d$ à partir de la proposition~\ref{pr-pira}.

\begin{lm}\label{lm-redfpol} Soient $\A$ une catégorie additive essentiellement petite, $\E$ une catégorie abélienne et $d\in\mathbb{N}$.
\begin{enumerate}
\item\label{it1} Si la multiplication par tout élément d'une partie multiplicative $S$ de $\mathbb{Z}$ est un isomorphisme dans $\E$, alors le foncteur canonique $\A\to\A[S^{-1}]$ induit une équivalence $\pol_d(\A[S^{-1}],\E)\xrightarrow{\simeq}\pol_d(\A,\E)$.
\item\label{it2} Si la multiplication par un entier $n>0$ est nulle dans $\E$, alors la réduction modulo $n^d$ dans $\A$ induit une équivalence $\pol_d(\A/n^d,\E)\xrightarrow{\simeq}\pol_d(\A,\E)$.
\end{enumerate}
\end{lm}

\begin{thm}\label{th-noethbis} Soit $A$ un anneau. On suppose que $k$ est noethérien et qu'il existe une $k$-algèbre commutative essentiellement de type fini $L$ telle que $A_k$ soit une $L$-algèbre finie. Alors tout foncteur polynomial de type fini de $\F(A,k)$ est noethérien.
\end{thm}

\begin{proof} L'hypothèse entraîne que l'algèbre $A_k$ est noethérienne (à gauche et à droite).

Supposons d'abord que la caractéristique de $k$ est une puissance non triviale d'un nombre premier $p$. L'anneau $A_k/p\simeq (A/p)\otimes_{\mathbb{Z}/p} (k/p)$, quotient de $A_k$, est  noethérien, donc  $A/p$ l'est également (lemme~\ref{lm-nfp}), et $A/p^N$ est noethérien pour tout $N\in\mathbb{N}$ (lemme~\ref{lm-relevatf}). Le théorème~\ref{th-princ} montre donc que la catégorie $\pol_d(A/p^N,k)$ est localement noethérienne pour tous $d, N\in\mathbb{N}$, de sorte qu'on conclut par le lemme~\ref{lm-redfpol}.\ref{it2}.

Supposons maintenant que $k$ est une $\mathbb{Z}[S^{-1}]$-algèbre telle que $k/k_\mathrm{tor}$ soit fidèlement plate, pour une partie multiplicative $S$ de $\mathbb{Z}$. Alors la noethérianité de $A_k$ entraîne celle de $A\otimes_\mathbb{Z}(k/k_\mathrm{tor})\simeq A[S^{-1}]\otimes_{\mathbb{Z}[S^{-1}]}(k/k_\mathrm{tor})$, donc de $A[S^{-1}]$ (lemme~\ref{lm-nfp}), de sorte que le théorème~\ref{th-princ} implique que la catégorie $\pol_d(A[S^{-1}],k)$ est localement noethérienne pour tout $d$. Le lemme~\ref{lm-redfpol}.\ref{it1} montre qu'il en est de même pour $\pol_d(A,k)$.

La conclusion découle des deux cas particuliers précédents et du lemme~\ref{lm-strZan}.
\end{proof}

\subsection{Quelques applications}

\begin{app}
Si $k$ est un \emph{corps} commutatif et $d\ge 1$ un entier, alors les assertions suivantes sont équivalentes :
\begin{enumerate}
    \item le corps $k$ est de type fini ;
    \item la catégorie $\pol_d(k,k)$ est localement noethérienne ;
    \item la catégorie $\pol_{d+1}(k,\mathbb{Z})$ est localement noethérienne ;
     \item la catégorie $\Pp_{d;\mathbb{Z}}(k,k)$ est localement noethérienne ;
    \item la catégorie $\Pp_{d+1;\mathbb{Z}}(k,\mathbb{Z})$ est localement noethérienne.
\end{enumerate}

En effet, comme un corps commutatif de type fini est un anneau commutatif essentiellement de type fini, la première assertion implique les autres par les théorèmes~\ref{th-pstrnoeth} et~\ref{th-princ}. La réciproque se déduit du fait que si $k$ est un corps commutatif et $n\ge 2$ un entier, alors l'anneau $k^{\otimes_\mathbb{Z}n}$ n'est noethérien que si $k$ est un corps de type fini, ainsi que des observations générales suivantes.

Si la catégorie $\pol_d(A,k)$ est localement noethérienne, alors la catégorie quotient $\pol_i(A,k)/\pol_{i-1}(A,k)$ est localement noethérienne pour tout entier $i\le d$, d'où l'on déduit facilement en utilisant la proposition~\ref{pr-pira} que la $k$-algèbre $A_k^{\otimes i}$ est noethérienne à gauche. L'analogue de la proposition~\ref{pr-pira} pour les foncteurs polynomiaux stricts montre de même que, si $\Pp_{d;k}(A,B)$ est localement noethérienne, où $A$ et $B$ sont des $k$-algèbres, alors $A^{\otimes d}\otimes B$ est noethérienne à gauche.
\end{app}

\begin{rem}Lorsque $d!$ est inversible dans $A_k$, le recollement de Pirashvili devient trivial, de sorte qu'il y a \emph{équivalence} entre la noethérianité locale de $\pol_d(A,k)$ et le caractère noethérien à gauche de $A_k^{\otimes i}$ pour tout $i\le d$. Nous ignorons si c'est encore le cas si l'on omet l'hypothèse d'inversibilité de $d!$ dans $A_k$. De surcroît, le problème de la préservation de la noethérianité par produit tensoriel pour des $k$-algèbres (même commutatives et lorsque $k$ est un corps) est notoirement difficile. Ainsi, trouver une condition nécessaire et suffisante simple sur l'anneau $A$ pour que $\pol_d(A,k)$ soit localement noethérienne semble hors de portée.
\end{rem}

\begin{app}
Nos résultats permettent d'améliorer \cite[cor.~6.22]{DT-schw}, relatif à la propriété de finitude homologique suivante pour un anneau $A$ :
\begin{itemize}
\item[(HGLF)] Pour tout $n\in\mathbb{N}$, le groupe abélien $H_n(\underset{i\in\mathbb{N}}{\col}\operatorname{GL}_i(A);\mathbb{Z})$ est de type fini.
\end{itemize}
En effet, la même démonstration que celle de \cite{DT-schw}, combinée au théorème~\ref{th-princ}, montre que si $A$ est une algèbre finie sur un anneau commutatif essentiellement de type fini et $I$ un idéal bilatère nilpotent de $A$ \emph{dont le groupe additif sous-jacent est de type fini}, et que l'anneau $A/I$ vérifie la propriété (HGLF), alors $A$ vérifie également (HGLF).

Par exemple, si $R$ est une algèbre finie sur un anneau commutatif essentiellement de type fini, que $R$ vérifie (HGLF) et que $V$ un $(R,R)$-bimodule dont le groupe abélien sous-jacent est de type fini, alors l'extension de carré nul $R\ltimes V$ vérifie (HGLF).
\end{app}

\bibliographystyle{plain}
\bibliography{bib-TJM.bib}

 \end{document}